\documentclass[10pt]{article}
\usepackage{amsmath,amssymb,amsthm}
\usepackage{epsfig}
\usepackage{color}

\title{New bounds on the signed total domination number of graphs}
\author {\\
S.M. Hosseini Moghaddam\thanks{corresponding author}\\
Qom Branch, Islamic Azad University\\
Qom, Iran\\
{\tt sm.hosseini1980@yahoo.com}\vspace{3mm}\\
D.A. Mojdeh\\
Department of Mathematics\\
University of Mazandaran, Babolsar, IRI\\
{\tt damojdeh@ipm.ir}\vspace{3mm}\\
Babak Samadi\\
Department of Mathematics\\
Arak University, Arak, IRI\\
{\tt b-samadi@araku.ac.ir}\vspace{3mm}\\
L. Volkmann\\
Lehrstuhl II f\"{u}r Mathematik\\
RWTH Aachen University, 52056 Aachen, Germany\\
{\tt volkm@math2.rwth-aachen.de}\\
}
\date{}

 \newtheorem{theorem}{Theorem}[section]

\newtheorem{lemma}[theorem]{Lemma}

\theoremstyle{definition}

\begin{document}

\maketitle

\begin{abstract}
\noindent In this paper, we study the signed total domination number in graphs and present new sharp lower and upper bounds for this parameter. For example by making use of the classic theorem of Tur\'{a}n \cite{t}, we present a sharp lower bound on $K‎_{r+1}‎$-free graphs for $r‎\geq2‎$. Applying the concept of total limited packing we bound the signed total domination number of $G$ with $‎\delta(G)‎\geq3‎‎$ from above by $n-2‎\lfloor‎\frac{2‎\rho‎_{o}(G)‎‎+‎\delta-3‎}{2}‎‎\rfloor‎‎‎$. Also, we prove that $\gamma_{st}(T)‎\leq n-2(s-s')‎$ for any tree $T$ of order $n$, with $s$ support vertices and $s'$ support vertices of degree two. Moreover, we characterize all trees attaining this bound.
\vspace{3mm}
\end{abstract}

\section{Introduction}
Let $G=(V, E)$ be a graph with vertex set $V=V(G)$ of order $n$ and edge set $E=E(G)$. The minimum and maximum degree of $G$ are denoted by $\delta=\delta(G)$ and $\Delta=\Delta(G)$, respectively. For a vertex $v\in V$, $N(v)$ is the open neighborhood of $v$, which is the set of vertices adjacent to $v$ and $N[v]=N(v)\cup \{v \}$ is the closed neighborhood of $v$. Let $[A,B]$ be the set of of edges with end points in both $A$ and $B$. The set of leaves and support vertices of a tree $T$ are denoted by $L(T)$ and $S(T)$, respectively. Also, we consider $L‎_{u}‎$ as the set of all leaves adjacent to the support vertex $u$. We use \cite{we} as a reference for terminology and notation which are not defined here.\\
A set  $S\subseteq V$ is a {\em total dominating set} if each vertex in $V$ is adjacent to at least one vertex in $S$. The {\em total domination number} $\gamma_{t}(G)$ is the minimum cardinality of a total dominating set. In \cite{hs}, Henning and Slater studied  the concept of {\em open packing} in graphs. A subset $B \subseteq V(G)$ is an {\em open packing} in $G$ if for every distinct vertices $u,v‎\in B‎$, $N(u)‎‎\cap N(v)=‎\emptyset‎‎‎$. The {\em open packing number}, $‎\rho‎_{o}(G)‎‎$, is the maximum cardinality of a an open packing in $G$.\\
A generalization of total domination titled {\em$k$-tuple total domination} have been studied by Henning and Kazemi in \cite{hk} (this concept had been studied by Zhao et al. \cite{zwx} as {\em total $k$-domination}). A subset $S\subseteq V$ is a {\em$k$-tuple total dominating set} in $G$ if $|N(v)\cap S|\geq k$, for all $v\in V(G)$. The {\em$k$-tuple total domination number}, $\gamma_{\times k, t}(G)$, is the smallest number of vertices in a $k$-tuple total dominating set.\\
Gallant et al. \cite{gghr} introduced the concept of {\em limited packing}. They exhibited some real-world applications of it to network security, NIMBY, market saturation and codes‎. In fact, A set of vertices $B \subseteq V$ is called a {\em$k$-limited packing set} in $G$ provided that for all $v\in V(G)$, we have $|N[v]\cap B|\leq k$. The {\em$k$-limited packing number}, denoted $L_{k}(G)$, is the largest number of vertices in a $k$-limited packing set. We can consider the concept of limited packing as dual of tuple domination in a graph. For more information the reader can consult \cite{msh}.\\
The above discussions give us a motivation to introduce the concept of {\em total limited packing} in graphs. Let $G$ be a graph, and $k\in N$. A set of vertices $L\subseteq V(G)$ is called a {\em$k$-total limited packing} in $G$ provided that for all $v\in V(G)$, we have $|N(v) \cap L|\leq k$. The {\em$k$-total limited packing number}, denoted $L_{k,t}(G)$, is the largest number of vertices in a $k$-total limited packing set. We can consider total limited packing first as a generalization of open packing, second as a dual of tuple total domination and third as a total version of limited packing. In fact, one can apply total limited packing to the subjects that we consider for limited packing as applications.\\
Let $S \subseteq V$. For a real-valued function $f:V\rightarrow R$ we define $ f(S)=\sum_{v \in S }f(v)$. Also, $f(V)$ is the weight of $f$. A signed total dominating function, abbreviated STDF, of $G$ is defined in \cite{z} as a function $f:V\rightarrow \{-1,1\}$ such that $f(N(v))\geq1$ for every $v\in V$. The signed total domination number, abbreviated  STDN, of $G$ is
$\gamma_{st}(G)=min\ f(V)|f \ is\ a \ STDF\\ \ of \ G \}$. There exist some real-world applications of signed total domination. For example, the author in \cite{h} applied this concept to model networks of people or organizations in which global decisions must be made.\\
In this paper, we continue the study of the concept of the signed total domination in graphs. The authors noted that most of the existing bounds on $\gamma_{st}(G)$ are lower bounds. In Section 1, we prove that $‎\gamma‎_{st}(G)‎‎‎\leq n-2‎\lfloor‎\frac{2‎\rho‎_{o}(G)‎‎+‎\delta-3‎}{2}‎‎\rfloor‎‎‎$, for a graph $G$ of order $n$ with $‎\delta(G)‎\geq3‎‎$. Also, we show that $n-2(s-s')$ is an upper bound of the signet total domination number of a tree $T$ of order $n$ with $s$ support vertices and $s'$ support vertices of degree two. Furthermore, we characterize all trees attaining this bound. In Section 2, we give some lower bound on the signed total domination number of graphs. As an application of the well-known theorem of Tur\'{a}n \cite{t}  we give a lower bound on this parameter for $K‎_{r+1}‎$-free graphs and conclude the lower bounds given in \cite{sc} for $r$-partite graphs and in \cite{sc} for triangle-free graphs as special cases.

\section{Upper bounds}

First we apply the concept of total limited packing to obtain a sharp upper bound on $\gamma_{st}(G)$.
\begin{theorem}
Let $G$ be a graph of order $n$ and $\delta\geq3$. Then
\begin{equation*}
‎\gamma‎_{st}(G)‎‎‎\leq n-2‎\lfloor‎\frac{2‎\rho‎_{o}(G)‎‎+‎\delta-3‎}{2}‎‎\rfloor‎‎‎
\end{equation*}
and this bound is sharp.
\end{theorem}
\begin{proof}
Let $L$ be a maximum $\lfloor\frac{\delta-1}{2}\rfloor$-total
limited packing set in $G$. Define $f:V\longrightarrow\{-1,1\}$, by
\begin{equation*}
f(v) = \left\{
\begin{array}{rl}
-1 & \text{if } v\in L\\
1 & \text{if } v\in V-L\\
\end{array} \right.
\end{equation*}
Since $L$ is a $\lfloor\frac{\delta-1}{2} \rfloor$-total limited
packing, $|N(v)\cap(V-L)|=deg(v)-|N(v)\cap L|\geq\delta-\lfloor
\frac{\delta-1}{2}\rfloor$. Therefore, for every vertex $v\in V$, we
have $f(N(v))=|N(v)\cap(V-L)|-|N(v)\cap L|\geq
\delta-\lfloor\frac{\delta-1}{2} \rfloor-\lfloor\frac{\delta-1}
{2}\rfloor\geq1$. Therefore $f$ is an STDF of $G$ with weight
$n-2|L|=n-2L_{\lfloor\frac{\delta-1}{2} \rfloor,t}(G)$. This shows that
\begin{equation}
\gamma_{st}(G)\leq n-2L_{\lfloor \frac{\delta-1}{2}\rfloor,t}(G).
\end{equation}
Now let $L$ be a maximun $\lfloor\frac{\delta-1}{2} \rfloor$-total limited
packing in $G$.ٌ We claim that $L\neq V$. If $L=V$ and $u\in V$
such that $deg(u)=\Delta$, then $\Delta=|N(u)\cap L|\leq
\lfloor\frac{\delta-1}{2} \rfloor$, a contradiction. Now let $u\in
V-L$. It is easy to check that $|N(v)\cap(L\cup\{u\})|\leq
\lfloor\frac{\delta-1}{2} \rfloor+1$ for each $v\in V$. Therefore
$L\cup\{u\}$ is a $(\lfloor \frac{\delta-1}{2}\rfloor+1)$-total
limited packing in $G$. Hence
$1+L_{\lfloor\frac{\delta-1}{2} \rfloor,t}(G)=|L\cup\{u\}|\leq L_{(\lfloor\frac{\delta-1}{2} \rfloor+1), t}(G)$.
If we continue this process we finally arrive at
$$L_{\lfloor\frac{\delta-1}{2} \rfloor,t}(G)\geq1+L_{(\lfloor\frac{\delta-1}{2} \rfloor-1),t}(G)
\geq \ldots \geq \left\lfloor\frac{\delta-1}{2}\right \rfloor-1+L_{1,t}(G),$$
and since $L_{1,t}(G)=\rho_{\circ}(G)$, we have
\begin{equation}
L_{\lfloor\frac{\delta-1}{2} \rfloor,t}(G)\geq
\left\lfloor\frac{\delta-1}{2}\right \rfloor-1+\rho_{\circ}(G).
\end{equation}
Inequalities $(1)$ and $(2)$ give the desired upper bound. Moreover,
this bound is sharp. It is sufficient to consider the complete graph
$K_{2n}$, when $n\geq2$. It is easy to check that
$\rho_{\circ}(K_{2n})=1$ and $\gamma_{st}(K_{2n})=2$.
\end{proof}

As the reader can check, the signed total domination number of a tree could be arbitrary large. Henning \cite{h} characterized all trees $T$ of order $n$ satisfying $\gamma_{st}(T)=n$ as all trees $T$ in which every vertex is a support vertex or is adjacent to a vertex of degree two. Now, we bound the signed total domination number of a tree from above by considering the number of its support vertices and characterize all trees attaining this bound. For this purpose, we define $‎\Omega‎$ to be the family of all trees $T$ satisfying:\vspace{2mm}\\
(a) For any support vertex $u$ with $|L‎_{u}‎|‎\geq2‎$, $deg(u)‎\leq4‎$;\\
(b)\ Every vertex in $V(T)‎\setminus(S(T)‎\cup L(T)‎)‎$ is adjacent to a support vertex or a vertex of degree two.
\begin{theorem}
Let $T$ be a tree of order $n$, with $s$ support vertices and $s'$ support vertices of degree two. Then
$$\gamma_{st}(T)‎\leq n-2(s-s')‎$$ with equality if and only if $T‎\in ‎‎‎\Omega‎‎$.
\end{theorem}
\begin{proof}
Let $S'$ be the set of all support vertices of degree two. We choose just one of the leaves adjacent to $v$, for every support vertex $v‎\in S‎\setminus S'‎‎$ and consider $L'$ as the set of all those leaves. We define $f:V(T)‎\rightarrow\{-1,1\}‎$ by $f(v)=-1$ if $v‎\in L'‎$, and $f(v)=1$ if $v‎\in V(T)‎\setminus L'‎‎$. It is easy to see that $f$ is a STDF of $G$. Thus
$$\gamma_{st}(T)‎‎\leq f(V(T))=n-2|L'|=n-2(s-s').‎$$
Suppose now that $f:V(T)‎\rightarrow\{-1,1\}‎$ is a minimum STDF of $T$ with weight $n-2(s-s')‎$. Let $u$ be a support vertex with degree at least three. Since $f$ is a minimum STDF then there must be at least one vertex $v$ in $N(u)$ with weight $-1$ under $f$. Without loss of generality we assume that $v$ belongs to $L'$. We first show that $T$ satisfies (a). Suppose that there is a support vertex $u$ with $|L‎_{u}‎|‎\geq2‎$ in which $deg(u)‎\geq5‎$ and let $v‎\in L'‎\cap N(u)‎‎‎$. Since $f(N(v))‎\geq1‎$, then $f(u)=1$. On the other hand, since $f$ is a minimum STDF of $T$ and $deg(u)‎\geq5‎$ there exists a vertex $w$ in $N(u)$, different from $v$, with weight $-1$ under $f$. This is a contradiction, because of $f(V(T))‎\leq n-2(s-s')-1=\gamma_{st}(T)-1‎$.\\
We now show that $T$ satisfies (b). Let there exists a vertex $v‎\in V(T)‎\setminus(S(T)‎\cup L(T)‎)‎‎$, in which is not adjacent to support vertices and all of its neighbors are of degree at least three. In this case, $f':V(T)‎\rightarrow\{-1,1\}‎$ that assigns to $v$ the value $-1$ and to all other vertices $w$ the value $f(v)$ would be a STDF of $G$ with weight $f'(V(T))‎\leq f(V(T))-2‎$. This is a contradiction.\\
We now let $T‎\in ‎\Omega‎‎$ and $f:V(T)‎\rightarrow\{-1,1\}‎$ be a minimum STDF of $T$ in which $f(v)=-1$, for all $v‎\in L'‎$. Suppose to the contrary that $f(V(T))<n-2(s-s')$. Therefore, there exists a vertex $v‎\in V(T)‎\setminus(S(T)‎\cup L'‎)‎$‎ such that $f(v)=-1$. Let $v‎\in L(T)‎$, then $v‎\in N(u)‎$ for some $u‎\in S‎\setminus S'‎‎$. Therefore $|L‎_{u}‎|‎\geq2‎$ and (a) implies $f(N(u))‎\leq0‎$ that is a  contradiction. If $v‎\in V(T)‎\setminus(S(T)‎\cup L(T)‎)‎$ then $v$ is not adjacent to a support vertex, similarly. Therefore, $v$ is adjacent to a vertex $w$ of order two, by (b). Hence, $f(N(w))‎\leq0‎$ that is a contradiction. The above discussion implies $\gamma_{st}(T)=n-2(s-s')‎$.
\end{proof}
We conclude this section by establishing an upper bound on the signed total domination number of a connected cubic graph. Henning \cite{h} proved that for every cubic graph of order $n$, $\gamma_{st}(G)‎\leq5n/7‎$. We now show that if $G$ is a connected cubic graph different from the Heawood graph $G‎_{14}‎$, then $\gamma_{st}(G)‎\leq2n/3$.
\begin{center}
\includegraphics{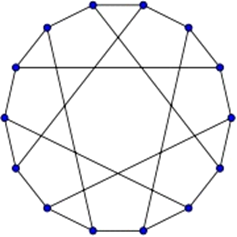}\\
Figure 1. {Heawood graph}
\end{center}

We need the following useful lemma.
\begin{lemma}\cite{hy}
If $G$ is a connected cubic graph of order $n$ different from the Heawood graph, then $‎\gamma‎_{‎\times‎2,t}(G)‎\leq 5n/6‎‎‎$, and this bound is sharp.
\end{lemma}
\begin{theorem}
If $G$ is a connected cubic graph of order $n$ different from the Heawood graph, then $‎\gamma‎_{st}(G)‎\leq 2n/3$, and this bound is sharp.
\end{theorem}
\begin{proof}
Let $f:V‎\rightarrow\{-1,1\}‎$ be minimum STDF of $G$. Since, $f(N(v))‎\geq1‎$, it follows that $|N(v)‎\cap V‎_{+}‎‎|‎‎\geq‎2‎‎‎$ for every vertex $v‎\in V(G)‎$. Hence, $V‎^{+}‎$ is a double total dominating set in $G$. Therefore,
\begin{equation}
(n+\gamma_{st}(G))/2=|V‎^{+}|‎\geq ‎‎\gamma‎_{‎\times‎2,t}(G)‎.
\end{equation}
Now let $D$ be a minimum double total dominating set in $G$. We define $f:V\longrightarrow\{-1,1\}$, by $f(v)=1$, if $v‎\in D‎$ and $f(v)=-1$, if $v‎\in V‎\setminus D‎‎$. Then, $f(N(v))=|N(v)‎\cap D‎|-|N(v)‎\cap (V‎\setminus D‎)‎|‎\geq1‎$. Therefore $f$ is a STDF of $G$. Thus,
\begin{equation}
\gamma_{st}(G)‎\leq ‎f(V(G))=2|D|-n=2\gamma‎_{‎\times‎2,t}(G)-n.
\end{equation}
Together inequalities (3) and (4) imply, $\gamma_{st}(G)=2\gamma‎_{‎\times‎2,t}(G)-n$. Now Lemma 2.2 implies the desired upper bound.
\end{proof}
\section{Lower bounds}
At this point we are going to present some sharp lower bounds on
$\gamma_{st}(G)$. First, let us introduce some notations. Let
$f:V‎\rightarrow‎\{-1,1\}$ be a minimum STDF of $G$. We define
$V_{+}=\{v\in V|f(v)=1\}$, $V_{-}=\{v\in V|f(v)=-1\}$,
$G_{+}=G[V_{+}]$ and $G_{-}=G[V_{-}]$ where $G_{+}$ and $G_{-}$ are
the subgraph of $G$ induced by $V_{+}$ and $V_{-}$, respectively.
For convenience, let $E_{+}=|E(G_{+})|$, $E_{-}=|E(G_{-})|$ and we define $V_{\circ}$ and $V_{e}$ as
the set of vertices with odd degree and the set of vertices with
even degree, respectively. Also $V_{+}^{o}=V_{+}\cap V_{o}$,
$V_{+}^{e}=V_{+}\cap V_{e}$, $V_{-}^{o}=V_{-}\cap V_{o}$ and
$V_{-}^{e}=V_{-}\cap V_{e}$. Finally, $deg_{G_{+}}(v)=|N(v)\cap
V_{+}|$ and
$deg_{G_{-}}(v)=|N(v)\cap V_{-}|$.\\\\
We begin with the following useful lemma.
\begin{lemma}
Considering the above notations, the following statements hold.\\\\
$(a)$\ \ \ $(\lfloor\frac{\delta}{2}
\rfloor+1)|V_{-}|\leq|[V_{+},V_{-}]|\leq(\lceil\frac{\Delta}{2}
\rceil-1)|V_{+}|$,\\
$(b)$\ \ \
$n+|V_{-}|+4|E_{-}|+|V_{e}|\leq2|E_{+}|+|[V_{+},V_{-}]|$.
\end{lemma}
\begin{proof}
(a) Let $f$ be a minimum STDF of $G$. Let $v\in V_{-}$. Since $f(N(v))\geq1$, we have
$deg_{G_{+}}(v)\geq \lfloor \frac{deg(v)}{2}\rfloor+1\geq \lfloor\frac{\delta}{2} \rfloor+1$.
Therefore $|[V_{+},V_{-}]|\geq(\lfloor\frac{\delta}{2}\rfloor+1)|V_{-}|$. Now let $v\in V_{+}$.
Then $deg_{G_{-}}(v)\leq \lceil\frac{deg(v)}{2} \rceil-1\leq \lceil\frac{\Delta}{2}\rceil-1$. Therefore
$|[V_{+},V_{-}]|\leq(\lceil\frac{\Delta}{2} \rceil-1)|V_{+}|$.

$(b)$ Let $f$ be a minimum STDF of $G$. First we derive a lower bound on $|[V_{+},V_{-}]|$.
Let $v\in V_{-}$. Since $f(N(v))\geq1$, we observe that $deg_{G_-}(v)\le |[v,V_+]|-1$ and
$deg_{G_-}(v)\le |[v,V_+]|-2$ when $deg(v)$ is even. This leads to
\begin{eqnarray*}
2|E_-| & = & \sum_{v\in V_-}deg_{G_-}(v) =\sum_{v\in V_-^e}deg_{G_-}(v)+ \sum_{v\in V_-^o}deg_{G_-}(v)\\
& \le & \sum_{v\in V_-^e}(|[v,V_+]|-2)+ \sum_{v\in V_-^o}(|[v,V_+]|-1)=|[V_-,V_+]|-|V_-|-|V_-^e|.
\end{eqnarray*}
This implies
\begin{equation}
|[V_{+},V_{-}]|\geq |V_{-}|+2|E_{-}|+|V_{-}^{e}|.
\end{equation}
Now let $v\in V_{+}$. Since $f(N(v))\geq1$, we observe that $deg_{G_+}(v)\ge |[v,V_-]|+1$ and
$deg_{G_+}(v)\ge |[v,V_-]|+2$ when $deg(v)$ is even. It follows that
\begin{eqnarray*}
2|E_+| & = & \sum_{v\in V_+}deg_{G_+}(v) =\sum_{v\in V_+^e}deg_{G_+}(v)+ \sum_{v\in V_+^o}deg_{G_+}(v)\\
& \ge & \sum_{v\in V_+^e}(|[v,V_-]|+2)+ \sum_{v\in V_+^o}(|[v,V_-]|+1)=|[V_+,V_-]|+|V_+|+|V_+^e|
\end{eqnarray*}
and so
\begin{equation}
|[V_{+},V_{-}]|\le 2|E_{+}|-|V_{+}^{e}|-|V_+|.
\end{equation}
Combining (5) and (6), we obtain $2|E_+|\ge n+2|E_-|+|V_e|$. Using this inequality and (5), we deduce that
$$2|E_+|+|[V_-,V_+]|\ge 2|E_-|+n+|V_e|+|V_-|+2|E_-|+|V_-^e|,$$
and this yields (b) immediately.
\end{proof}
We are now in a position to present the following lower bounds.
\begin{theorem}
Let $G$ be a graph of order $n$ and size $m$. Then\\\\
$(i)$  $\gamma_{st}(G)\geq \frac{(\lfloor
\frac{\delta}{2}\rfloor-\lceil \frac{\Delta}
{2}\rceil+2)n}{\lfloor \frac{\delta}{2}\rfloor+\lceil\frac{\Delta}{2} \rceil}$,\\
$(ii)$  $\gamma_{st}(G)\geq \frac{(5-3\Delta
-2\lceil\frac{\Delta}{2}\rceil)n+2n_{e}+8m}
{3\Delta+2\lceil\frac{\Delta}{2}\rceil-1}$,\\
$(iii)$  $\gamma_{st}(G)\geq
\frac{(5+3\delta-2\lceil\frac{\Delta}{2}\rceil)n+2n_{e}-4m}
{3\delta+2\lceil\frac{\Delta}{2}\rceil-1}$,\\
and these bounds are sharp, where $n_{e}=|V_{e}|$.
\end{theorem}
\begin{proof}
$(i)$ Using  $(a)$ of the previous lemma and  $|V_{+}|=\frac{n+\gamma_{st}(G)}{2}$ and
$|V_{-}|=\frac{n-\gamma_{st}(G)}{2}$, the desired bound is easy to verify.

$(ii)$ Using $4m=4|E_{-}|+4|[V_{+},V_{-}]|+4|E_{+}|$ and Lemma 3.1 (b), we deduce that
$$n+|V_{-}|+|V_{e}|+4m\leq 6|E_{+}|+5|[V_{+},V_{-}]|.$$
Applying
$$2|E_+| = \sum_{v\in V_+}deg_{G_+}(v) =\sum_{v\in V_+}(deg(v)-|[v,V_-]|)\le\Delta|V_+|-|[V_+,V_-]|,$$
we obtain, $n+|V_{-}|+|V_{e}|+4m\leq 3\Delta|V_{+}|+2|[V_{+},V_{-}]|$. By Lemma 3.1 (a),
we have $2|[V_{+},V_{-}]|\leq 2(\lceil \frac{\Delta}{2}\rceil-1)|V_{+}|$ and therefore
$$n+|V_{-}|+|V_{e}|+4m\leq (3\Delta+2\left\lceil\frac{\Delta}{2}\right \rceil-2)|V_{+}|.$$
Because of $|V_{+}|=\frac{n+\gamma_{st}(G)}{2}$ and $|V_{-}|=\frac{n-\gamma_{st}(G)}{2}$, we obtain
the desired bound.\\
$(iii)$ Using $2m=2|E_{-}|+2|[V_{+},V_{-}]|+2|E_{+}|$ and Lemma 3.1 (b), we conclude that
$$n+|V_{-}|+6|E_{-}|+|V_{e}|+|[V_{+},V_{-}]|\leq2m.$$
Applying
$$2|E_-| = \sum_{v\in V_+}deg_{G_-}(v) =\sum_{v\in V_-}(deg(v)-|[v,V_+]|)\ge\delta|V_-|-|[V_-,V_+]|,$$
we see that $n+(1+3\delta) |V_{-}|+|V_{e}|-2|[V_{+},V_{-}]|\leq2m$. By Lemma 3.1 (a), we
have $2|[V_{+},V_{-}]|\leq 2(\lceil \frac{\Delta}{2}\rceil-1)|V_{+}|$ and thus
$$n+(1+3\delta)|V_{-}|+|V_{e}|-2(\left\lceil\frac{\Delta}{2}\right\rceil-1)|V_{+}|\leq2m.$$
This implies the last bound.\\
Since $\gamma_{st}(C_{n})=n$1, the cycle $C_n$ attains all the three lower bounds.
\end{proof}
We note that that the bound (i) in Theorem 3.2 can be found in the paper \cite{h} by Henning. However, our proof
is much shorter and transparent. If $G$ is an $r$-regular graph of order $n$, then (i) leads to
$\gamma_{st}(G)\geq2n/r$ when $r$ is even and $\gamma_{st}(G)\geq n/r$ when $r$ is odd. This is
a result by Zelinka \cite{z}.

\begin{theorem}
For every graph $G$ of order $n$, minimum degree $‎\delta‎$ and maximum degree $‎\Delta‎$,
$$‎\gamma‎_{st}(G)‎\geq-n+2max\{‎\lceil‎‎\frac{‎\Delta‎+3}{2}‎\rceil‎‎,‎\lceil‎‎\frac{2‎\gamma‎_{t}(G)+‎\delta-1‎‎‎}{2}‎\rceil‎‎\}‎‎‎$$
and this bound is sharp.
\end{theorem}
\begin{proof}
We first prove the following claims.\\
{\bf Claim 1.} $‎\gamma‎_{st}(G)‎\geq-n+2‎\lceil‎‎\frac{‎\Delta‎+3}{2}‎\rceil‎$.\\
Let $f:V‎\rightarrow\{-1,1\}‎$ be minimum STDF of $G$. Since, $f(N(v))‎\geq1‎$, it follows that $|N(v)‎\cap V‎_{-}‎‎|‎\leq ‎\lfloor‎\frac{‎\Delta‎-1}{2}‎‎\rfloor‎‎‎$ for every vertex $v‎\in V(G)‎$. Therefore $V‎_{-}‎$ is a $‎\lfloor‎\frac{‎\Delta‎-1}{2}‎‎\rfloor‎‎‎$-total limited packing in $G$. Thus
\begin{equation}
(n-‎\gamma‎_{st}(G)‎)/2=|V‎_{-}‎|‎\leq L‎_{‎\lfloor‎\frac{‎\Delta‎-1}{2}‎‎\rfloor‎‎‎,t}(G)‎‎.
\end{equation}
On the other hand, similar to the proof of Theorem 2.1 and taking into account that $L‎_{‎\Delta‎+1,t}(G)=n$, we obtain
\begin{equation*}\label{EQ6}
\begin{array}{lcl}
L‎_{‎\lfloor‎\frac{‎\Delta‎-1}{2}‎‎\rfloor‎‎‎,t}(G)‎&\leq & ‎L‎_{‎\lfloor‎\frac{‎\Delta‎-1}{2}‎‎\rfloor+1‎‎‎,t}(G)-1‎\leq...\\
‎&\leq & L‎_{‎\Delta‎+1,t}(G)‎‎‎-‎\Delta+‎\lfloor‎\frac{‎\Delta‎-3}{2}‎‎\rfloor=n-‎\lceil‎\frac{‎\Delta‎+3}{2}‎‎\rceil‎‎‎‎‎.
\end{array}
\end{equation*}
Now inequality (5) implies $‎\gamma‎_{st}(G)‎\geq-n+2‎\lceil‎‎\frac{‎\Delta‎+3}{2}‎\rceil‎$, as desired.\\
{\bf Claim 2.} $‎\gamma‎_{st}(G)‎\geq-n+2‎\lceil‎\frac{2‎\gamma‎_{t}(G)+‎\delta-1‎‎‎}{2}‎‎\rceil‎‎$.\\
Since, $f(N(v))‎\geq1‎$, it follows that $|N(v)‎\cap V‎‎_{+}‎‎|‎\geq ‎\lceil‎\frac{‎\delta‎+1}{2}‎‎\rceil‎‎‎$ for every vertex $v‎\in V(G)‎$. Therefore $V‎_{+}‎$ is a $\lceil‎\frac{‎\delta‎+1}{2}‎‎\rceil‎‎‎$-tuple total dominating set in $G$. Thus
\begin{equation}
(n+‎\gamma‎_{st}(G)‎)/2=|V‎_{+}‎|‎‎\geq ‎\gamma‎_{‎\times‎\lceil‎\frac{‎\delta‎+1}{2}‎‎\rceil‎‎‎,t}(G)‎‎‎.
\end{equation}
Now let $D$ be a minimum $‎\lceil‎\frac{‎\delta‎+1}{2}‎‎\rceil$-tuple total dominating set in $G$. Then $|N(v)‎\cap D‎|‎\geq \lceil‎\frac{‎\delta‎+1}{2}‎‎\rceil‎$, for every vertex $v\in V(G)$. Let $u\in D$. It is easy to see that $|N(v)‎\cap (‎D‎\setminus\{u\})‎‎‎|‎\geq ‎\lceil‎\frac{‎\delta‎+1}{2}‎‎\rceil‎-1$, for all $v\in V(G)$. Hence, $‎D‎\setminus\{u\}$ is a $‎(\lceil‎\frac{‎\delta‎+1}{2}‎‎\rceil-1)$-tuple total dominating set in $G$. Therefore, $‎\gamma‎_{‎\times‎‎\lceil‎\frac{‎\delta‎+1}{2}‎‎\rceil‎‎,t}(G)‎‎‎\geq ‎\gamma‎_{‎\times(‎‎\lceil‎\frac{‎\delta‎+1}{2}‎‎\rceil-1)‎‎,t}(G)+1‎$. By repeating this process, we obtain
\begin{equation*}\label{EQ6}
\begin{array}{lcl}
\gamma‎_{‎\times‎‎\lceil‎\frac{‎\delta‎+1}{2}‎‎\rceil‎‎,t}(G)&‎‎‎\geq & ‎\gamma‎_{‎\times(‎‎\lceil‎\frac{‎\delta‎+1}{2}‎‎\rceil-1)‎‎,t}(G)+1‎\geq...‎\\
&\geq & ‎\gamma‎_{‎\times‎1,t}(G)+‎\lceil‎\frac{‎\delta‎+1}{2}‎‎\rceil-1=‎\gamma‎_{t}(G)‎‎‎‎‎‎‎‎+\lceil‎\frac{‎\delta‎-1}{2}‎‎\rceil.
\end{array}
\end{equation*}
By inequality (6),
$$(n+‎\gamma‎_{st}(G)‎)/2‎\geq \gamma‎_{t}(G)‎‎‎‎‎‎‎‎+\lceil‎\frac{‎\delta‎-1}{2}‎‎\rceil‎.$$
This completes the proof of Claim 2.\\
The result now follows by Claim 1 and  Claim 2. The bound is sharp for the complete graph.
\end{proof}
A graph is $K_p$-free if it does not contain the complete graph $K_p$ as a subgraph. For our next
lower bound, we use the following well-known Theorem of Tur\'{a}n \cite{t}.
\begin{theorem}
If $G$ is a $K_{r+1}$-free graph of order $n$, then
$$|E(G)|\le \frac{r-1}{2r}\cdot n^2.$$
\end{theorem}

\begin{theorem}
Let $r\ge 2$ be an integer, and let $G$ be a $K_{r+1}$-free graph of order $n$.
If $c=\lceil(\delta(G)+1)/2\rceil$, then
$$\gamma_{st}(G)\ge \frac{r}{r-1}\left(-(c-1)+\sqrt{(c-1)^2+4\frac{r-1}{r}cn}\right)-n.$$
\end{theorem}
\begin{proof}
By Lemma 3.1 (a), we have
\begin{equation}
|[V_+,V_-]|\ge\left\lceil\frac{\delta(G)+1}{2}\right\rceil|V_-|=c|V_-|=c(n-|V_+|).
\end{equation}
Furthermore, Theorem 3.6 leads to
\begin{eqnarray*}
|[V_+,V_-]| & = & \sum_{v\in V_+}|N(v)\cap V_-|\le \sum_{v\in V_+}(|N(v)\cap V_+|-1)\\
& = & 2|E(G[V_+])|-|V_+|\le\frac{r-1}{r}|V_+|^2-|V_+|.
\end{eqnarray*}
Combining this inequality chain with (5), we obtain
$$(n-|V_+|)c\le\frac{r-1}{r}|V_+|^2-|V_+|$$
and thus
$$\frac{r-1}{r}|V_+|^2+(c-1)|V_+|-cn\ge 0.$$
It follows that
$$|V_+|\ge\frac{r}{2(r-1)}\left(-(c-1)+\sqrt{(c-1)^2+4\frac{r-1}{r}cn}\right),$$
and so we arrive at the desired bound
$$\gamma_{st}(G)=2|V_+|-n\ge\frac{r}{r-1}\left(-(c-1)+\sqrt{(c-1)^2+4\frac{r-1}{r}cn}\right)-n.$$
\end{proof}
For the special case that $G$ is an $r$-partite graph, Theorem 3.7 was proved by Shan and Cheng in \cite{sc}.
In that paper the authors have constructed $r$-partite graphs with equality in the inequality of
Theorem 3.7, and therefore this theorem is sharp. In addition, one can find Theorem 3.7 for triangle-free
graphs in \cite{sc}.


\end{document}